\documentclass[10pt]{amsart}
\usepackage{cancel}
\usepackage{amsmath,comment}
\usepackage{mathtools}
\usepackage{amssymb}
\usepackage{amsfonts}
\usepackage{amsthm}
\usepackage{mathrsfs}
\usepackage{polynom}
\usepackage{mathbbol}
\usepackage{latexsym,amscd}
\usepackage[all]{xy}
\usepackage{color,graphicx}

\newtheorem{thm}{Theorem}
\newtheorem{prop}{Proposition}
\newtheorem{lem}{Lemma}
\newtheorem{conjecture}{Conjecture}
\numberwithin{equation}{section}
\numberwithin{prop}{section}
\numberwithin{lem}{section}
\numberwithin{thm}{section}
\newtheorem{cor}{Corollary}
\numberwithin{cor}{section}

\theoremstyle{definition}

\numberwithin{defn}{section}

\newtheorem{rem}{Remark}
\numberwithin{rem}{section}

\newcommand{\ZZ}{\mathbb{Z}}
\newcommand{\CC}{\mathbb{C}}

\newcommand{\be}{\begin {equation}}
\newcommand{\ee}{\end{equation}}

\newcommand{\bee}{\begin {equation*}}
\newcommand{\eee}{\end{equation*}}

\newcommand{\nop}[1]{{}^{\scriptscriptstyle{\circ}}_{\scriptscriptstyle{\circ}}{#1}{}^{\scriptscriptstyle{\circ}}_{\scriptscriptstyle{\circ}}}

\def \F{\mathcal{F}}

\def \L{\Lambda}

\def \1{\mathbb{1}}
\def \V{\mathcal{V}}
\def \L{\mathcal{L}}

\usepackage{tikz-cd}

\begin{document}
\title{Permutation orbifolds of Virasoro vertex algebras and $W$-algebras}
%\noindent Keywords: Vertex operator algebra, invariant theory\\
%MSC Codes: 17B69, 16W22
\author{Antun Milas, Michael Penn, Christopher Sadowski}
\address{Department of Mathematics and Statistics, SUNY-Albany}
\email{amilas@albany.edu}
\address{Mathematics Department, Randolph College}
\email{mpenn@randolphcollege.edu}
\address{Department of Mathematics and Computer Science, Ursinus College}
\email{csadowski@ursinus.edu}
%\email{???}
%\email{michael.penn@coloradocollege.edu}

\maketitle

\begin{abstract} We study permutation orbifolds of the $2$-fold and $3$-fold tensor product for the Virasoro vertex algebra $\V_c$ of central charge $c$. In particular, we show that for all but finitely many central charges $\left(\V_c^{\otimes 3}\right)^{\mathbb{Z}_3}$ is a $W$-algebra of type $(2, 4, 5, 6^3 , 7, 8^3 , 9^3 , 10^2 )$. We also study orbifolds of their simple quotients and obtain new realizations of certain rational affine $W$-algebras associated to a principal nilpotent element. Further analysis of permutation orbifolds of the celebrated $(2,5)$-minimal vertex algebra $\L_{-\frac{22}{5}}$ is presented.
\end{abstract}

\section{Introduction}

%{\bf TODO 1: the paper can easily end up with Jim or Yi-Zhi so we have to explain the difference between 
%OPE $\partial^{(n)} x$ and $(L(-1)^n x)$ and explain that our notation is convenient because of OPE. Also we should %consider adding more general VOA references - Jim and Haisheng's book, etc.} 

%{\bf TODO 2: think about possible improvements - additional results?, etc.  }

%{\bf TODO 3: what about $c=0$?  }

For any vertex operator algebra $V$, the $n$-fold tensor product $V^{\otimes n}=V \otimes \cdots \otimes V$ has a natural vertex operator algebra structure.
The symmetric group $S_n$ acts on $V^{\otimes n}$ by permuting tensor factors and thus $S_n \subset {\rm Aut}(V^{\otimes n})$. The ${S_n}$-invariants subalgebra of $V^{\otimes n}$, denoted by $(V^{\otimes n})^{S_n}$ is called the {\em symmetric orbifold} of $V^{\otimes n}$. 
For any $n$-cycle $\sigma \in S_n$ we let $(V^{\otimes n})^{\mathbb{Z}_n}$ denote the invariant subalgebra under $\langle \sigma \rangle \cong \mathbb{Z}_n$, called {\em cyclic} or $n$-cyclic orbifold. There is a rich literature on this 
subject; see \cite{A, BDM,BHL,BV,DXY, DRX1,DRX2,MPSh,MPW} and reference therein.

We denote by $\V_c:=V_{Vir}(c,0)$, the universal Virasoro vertex operator algebra of central charge $c$ and $\L_c:=L(c,0)$ its unique simple quotient. It is well-known due to Feigin and Fuchs that 
$$\V_c \neq \L_c \ \Leftrightarrow \ c=c_{p,q}:=1-\frac{6(p-q)^2}{pq},$$
where $p,q \geq 2$ are coprime integers. Moreover $\L_{c_{p,q}}$ is lisse and rational if and only if $c=c_{p,q}$ \cite{Wang}. We also denote by $M(c,h)$ the Verma module of lowest weight $h$ and by $L(c,h)$ its simple quotient. The vertex algebra $\V_c$ is of fundamental important in physics, that is, it acts on conformal field theories with a specified central charge.

In this paper we investigate three types of permutation orbifolds: the $S_2$-orbifold of $\V_c^{\otimes 2}$ and of $\L_c^{\otimes 2}$, cyclic orbifolds of $\V_c^{\otimes 3}$ and $\L_c^{\otimes 3}$, and an exceptional
permutation orbifolds of $\L_{-\frac{22}{5}}^{\otimes 3}$. We leave analysis of general $S_3$-permutation orbifolds of $\V_c^{\otimes 3}$ for future work \cite{MPS2} as it involves additional computations and further analysis of several special cases. One of the upshots of this paper is that these orbifolds provide several infinite families of rational and lisse vertex algebras, which is a consequence of results in \cite{CarMiy,Miy}.

The paper is organized as follows. In Section 2, we discuss basic facts concerning Virasoro vertex algebra, tensor product construction of $\V_c^{\otimes n}$, and their permutation orbifolds. In Section 3, we consider the orbifold $(\V_c^{\otimes 2})^{S_2}$. Our main result here, Theorem \ref{mainrank2},  
shows that generically this orbifold is of type $(2,4,6,8)$. Section 4 is devoted to analysis of the simple orbifold
$(\L_c^{\otimes 2})^{S_2}$. In particular, we determined its type for special central charges $c_{2,5}, c_{2,7}, c_{3,4}, c_{3,5}$ and $c_{2,9}$ (see Theorem \ref{rank2simple}). In Section 5, using recent results of Kanade and Linshaw, we give another interpretation of several families of rational vertex algebras $(\L_c^{\otimes 2})^{S_2}$. This in particular proves rationality of certain affine $W$-algebras of symplectic type (cf. Corollary \ref{sp})  and the $\mathbb{Z}_2$-orbifold of affine $W$-algebras of orthogonal type (cf. Corollary \ref{so}). In Section 6 we prove the main result of this paper, Theorem \ref{Z3genreduce}, on the structure of most general cyclic orbifolds $(\V_c^{\otimes n})^{\mathbb{Z}_3}$. Section 7 contains a detailed analysis of two important special cases:
$(\L_{-\frac{22}{5}}^{\otimes 3})^{\mathbb{Z}_3}$ and $(\L_{-\frac{22}{5}}^{\otimes 3})^{S_3}$. In particular, we show they are of type $(2,5,6,9)$ and $(2,6)$, respectively. Then in Section 8, we analyze conformal embeddings among minimal central charges and finally, in Section 9, we outline several questions and directions for future work.

{\bf Acknowledgments.} We thank A. Linshaw for very useful comments  on Section 5. We also thank A. Dujella for 
help with numerical computation in Section 8. We also thank D. Adamovi\'c for discussion.

 \section{Permutation orbifolds of $\V_c^{\otimes n}$}

Consider $\V_c$, the universal Virasoro vertex algebra with central charge $c$ generated by the weight $2$ vector $\omega$ with vertex operator $Y(\omega,z)=L(z)=\sum_{n\in\ZZ}L(n)z^{-n-2}$ with 
\be
[L(m),L(n)]=(m-n)L(m+n)+\frac{c}{12}\delta_{m+n,0}\1,\ee
or equivalently, using OPE notation, we have 
\be
L(z)L(w)=\nop{L(z)L(w)}+\frac{\partial_w L(w)}{(z-w)}+\frac{2 L(w)}{(z-w)^2}+\frac{c/2}{(z-w)^4}.\ee
 
%so that $\H(n)=\langle \alpha_1(-1)\1, \cdots ,\alpha_n(-1)\1 \rangle$.
For convenience, we suppress the tensor product symbol and let  $$L_i(-m)\1:=\underbrace{\1 \otimes  \cdots  \1} _{(i-1)-{\rm factors}} \otimes L(-m) \1 \otimes \underbrace{ \1 \otimes \cdots \otimes \1}_{(n-1-i)-{\rm factors}} \in \V_c^{\otimes n},$$  such that
$\V_c^{\otimes n}=\langle L_1(-2)\1, \cdots ,L_n(-2)\1 \rangle$.
We denote by $\omega=\omega_1+ \cdots + \omega_n$ the conformal vector.
We consider the natural action of $S_n$ on $\V_c^{\otimes n}$ given by 
permuting tensor factors, that is 
\be
\sigma\cdot L_{i_1}(m_1)\cdots L_{i_k}(m_k)\1=L_{\sigma(i_1)}(m_1)\cdots L_{\sigma(i_k)}(m_k)\1,
\ee
for $1\leq i_j\leq n$, $m_j<-1$, and $\sigma\in S_n$. 

Next result is certainly known.
\begin{prop}
We have 
$$\text{\rm Aut } \V_c ^{\otimes n}=S_n.$$
\end{prop}
\begin{proof} It is clear that conformal vectors in $\V^{\otimes n}$ take form $\sum_{i \in S } \omega_i$, where $S \subseteq \{1,...,n \}$. Let $f \in \text{Aut }\V^{\otimes n}$. Then $f(\sum_{i \in S } \omega_i)=\sum_{i \in S' } \omega_i$, and therefore $f(\omega)=\omega$. This implies $f(\omega_i)=\omega_{\sigma(i)}$ for some $\sigma \in S_n$.
\end{proof}

For a vertex algebra $V$ denote by ${\rm gr}(V)$ the associated graded Poisson algebra of $V$ \cite{Ar2, HLi}. We have a natural linear isomorphism 
\be\label{linearisom}
\V_c^{\otimes n}\cong {\rm gr}(\V_c^{\otimes n}) \cong \CC[x_i(m) | 1\leq i \leq n, m\geq 0]\ee
induced by $L_i(-m-2)\mapsto x_i(m)$ for $m\geq 0$. 
In particular, the Poisson Zhu algebra of $\V_c^{\otimes n}$, denoted by $R_{\V^{\otimes n}}$, is isomorphic to $\mathbb{C}[x_1,...,x_n]$.
The right-hand is also isomorphic to $\mathbb{C}[J_\infty(\mathbb{A}^n)]$,  the ring of functions on the arc space of $\mathbb{A}^n$. 
%Moreover, the Poisson Zhu algebra of $\V_c^n$ is isomorphic to $\mathbb{C}[x_1,...,x_n]$.
%Using terminology in \cite{L2}, we say that $\CC[x_i(m) | 1\leq i \leq n, m\geq 0]$ is the associated graded algebra of %the vertex algebra $\V_c^{\otimes n}$. 
The Poisson algebra $\CC[x_i(m) | 1\leq i \leq n, m\geq 0]$ comes equipped with a derivation
%\color{red}: is it $x_i(m)\mapsto x_i(m+1)$? \color{black}
\be\begin{aligned}
\partial:\CC[x_i(m) | 1\leq i \leq n, m\geq 0]&\to \CC[x_i(m) | 1\leq i \leq n, m\geq 0]\\
x_i(m)&\mapsto x_i(m+1),\end{aligned}\ee
and where the action of $\partial$ is extended to the whole space via the Leibniz rule. 
This definition of $\partial$  is compatible with the translation operator in $\V_c^{\otimes n}$ given by $T(v)=v_{-2}\mathbb{1}$.
%makes the following diagram commute
%\[\begin{tikzcd}
%\CC[x_i(m) | 1\leq i \leq n, m\geq 0] \arrow{r}{\partial} \arrow{d}{} & \CC[x_i(m) | 1\leq i \leq n, m\geq 0] \arrow{d}{} \\
%\H(n) \arrow{r}{T} & \H(n)
%\end{tikzcd}
%\]
%where $T$ is the translation operator given by $T(v)=v_{-2}\mathbb{1}$ for $v\in\H(n)$, where the vertical maps are the linear isomorphism described above. As such, we say that $\CC[x_i(m) %1\leq i \leq n, m\geq 0]$ is the associated $\partial$-ring of the vertex algebra $\H(n)$. 
Then we have a standard result \cite{HLi,L2} (cf. \cite{Ar2}).
\begin{lem}\label{reconstruction}
Let $V$ be a vertex algebra with a ``good" $\mathbb{Z}_{\geq 0}$ filtration. If $\{\tilde{a}_i|i\in I\}$ generates $gr(V)$ then $\{a_i | i\in I\}$ strongly generates 
$V$, where $a_i$ and $\tilde{a}_i$ are related via the natural linear isomorphism described by the 
$\mathbb{Z}_{\geq 0}$ filtration.
\end{lem}

We say that $v \in V$ is {\em primary} of conformal weight $r$ if $L(n)v=0$, $n \geq 1$ and $L(0)v=r v$.
In our work we consider several vertex algebras strongly generated by the Virasoro vector and several primary fields
of conformal weight $r_1,...,r_k$. Using physics' terminology we say that such vertex algebra is a $W$-algebra of type
$(2,r_1,...,r_k)$ (see \cite{BS,BEHHH,EHH}).

%There is a wast literature on $W$-algebras constructed from primary vectors by physicists.

Further, we recall that the invariant ring $\CC[x_1,\dots,x_n]^{S_n}$ has a variety of generating sets, including the power sum polynomials 
\begin{align*}
& p_i=x_1^i+\cdots+x_n^i, \ \ i \geq 1.
\end{align*}
In addition to this, it is common to study the invariant theory of the ring of infinitely many commuting copies of this polynomial algebra, where we denote by $x_i(m)$ the copy of $x_i$ from the $m^{\text{th}}$ copy.  A  well-known theorem of Weyl \cite{W} shows that $\CC[x_i(m) | 1\leq i \leq n, m\geq 0]^{S_n}$ is generated by the polarizations of these polynomials
\be\label{dringgenerators}
q_k(m_1,\dots,m_k)=\sum_{i=1}^n x_i(m_1)\cdots x_i(m_k),\ee
for $1\leq k\leq n$.
Now, applying Lemma \ref{reconstruction}, we have an initial strong generating set for the orbifold $(\V_c^{\otimes n})^{S_n}$ given by the vectors
\be\label{firstorbifoldgenerators}
u_k(m_1,\dots,m_k)=\sum_{i=1}^n L_i(-2-m_1)\cdots L_i(-2-m_k)\1,
\ee
for $1\leq k\leq n$ and $m_j\geq 0$. Throughout we will switch between working directly in the setting of the vertex operator algebra $\V_c^{\otimes n}$ and its copy inside $(\text{End }\V_c^{\otimes n})[[z,z^{-1}]]$ via the vertex operator map $$Y(\cdot,z):\V_c(n)\to (\text{End }\V^{\otimes n}_c(n))[[z,z^{-1}]].$$ Under this map we have 
\be\begin{aligned}
U_k(m_1,\dots,m_k):=&Y(u_k(m_1,\dots,m_k),z)\\=&\frac{1}{(m_1-1)!}\cdots \frac{1}{(m_k-1)!}\sum_{i=1}^n\nop{\partial_z^{m_1}L_i(z)\cdots\partial_z^{m_k}L_i(z)},\end{aligned}\ee
where by $\nop{-}$ the normal ordered product and we will often suppress the formal variable $z$ and write $\partial^m L_i=\partial_z^m L_i(z)$.
%Set {\bf AM: this is confusing notation but I'll stick with it anyway. Also it's not consistently used. Sometimes $c$ denotes the total central charge. %Michael can you fix this?} $\V^c(n)=\left(V_{\text{Vir}}(c,0)^{\otimes n}\right)^{S_n}$ and $\V_c(n)$ its unique simple quotient.

%\color{red} Is this below clear enough? \color{black}
%We will also need another simple fact.
%\begin{lem} 
%Let $I \subseteq V$ be a maximal ideal and $G$ a finite subgroup of $Aut(V)$, then $(V/I)^G \cong V^G/I^G$. In %particular, 
%let $L$ denote the (unique) simple quotient of $(\mathcal{V}_c^{\otimes^n})^{S_n}$, then $L \cong %(\L_c^{\otimes^n})^{S_n}$.
%\end{lem}
%\begin{proof} It follows directly from exactness of the invariant functor $M \mapsto M^G$ for $G$ finite.
%\end{proof}

\subsection{Characters}

Since $\V_c$ is isomorphic (as graded vector spaces) to $U(Vir_{\leq -2})$, its graded dimension (or simply {\em character}) is given by 
$${\rm ch}[\V_c](\tau)=\frac{1}{\prod_{n \geq 2}(1-q^n)},$$
where $q=e^{2 \pi i \tau}$. For minimal central charges, the character of $\L_{c_{p,q}}$ is given by \cite{Wak}
$${\rm ch}[\L_{c_{p,q}}](\tau)=\frac{ \theta_{p,q}^{1,1} (\tau)}{\eta(\tau)},$$
where
$$\theta_{p,p'}^{1,1} (\tau)=\sum_{j \in
\mathbb{Z}}
\left(q^{{j(jpq+q-p)}}-q^{{(jp+1)(jq+1)}}  \right)$$
and $\eta(\tau)=q^{\frac{1}{24}} \prod_{i \geq 1} (1-q^i)$.
Equipped with this formula, it is easy to compute the character of cyclic and permutation orbifold algebras using conjugacy classes of $S_n$ \cite{MPSh}.
For special cases relevant for this work, we get (here $X_c=\V_c$ or $X_c=\L_c$): 
\begin{align*} 
& {\rm ch}[(X_c^{\otimes 2})^{S_2}](\tau)=\frac12 {\rm ch}[X_c]^2 (\tau)+\frac12 {\rm ch}[X_c] (2\tau), \\
& {\rm ch}[(X_c^{\otimes 3})^{\mathbb{Z}_3}](\tau)=\frac13 {\rm ch}[X_c]^3 (\tau)+\frac23 {\rm ch}[X_c] (3\tau), \\
& {\rm ch}[(X_c^{\otimes 3})^{S_3}](\tau)=\frac16 {\rm ch}[\V_c]^3 (\tau)+\frac12 {\rm ch}[X_c] (2 \tau) {\rm ch}[X_c] (\tau)+   \frac13 {\rm ch}[X_c] (3 \tau).
\end{align*}

\section{The orbifold $\left(\V_c^{\otimes 2}\right)^{S_2}$}

We first consider the symmetric oribofold of $\V_c^{\otimes 2}$. For $\omega_i$ with $i=1,2$, define the component operators $L(n)$ by 
$$L_i(z)=Y(\omega_i,z)=\sum_{n\in\mathbb{Z}}L_i(n)z^{-n-2}.$$
Consider the vectors 
\be\begin{aligned}
\omega&=\omega_1+\omega_2,\\
w_k&=L_1(2-k)L_1(-2)\1+L_2(2-k)L_2(-2)\1,\end{aligned}\ee
where the $w_k$ have conformal weight $k$. Further we define the component operators $J_k(n)$ by 
$$W^k(z)=Y(w_k,z)=\sum_{n\in\mathbb{Z}}J_k(n)z^{-n-k}.$$
By the classical invariant theory of $S_2$ it is clear that $\omega$ along with $w_k$ for $k\geq 0$ strongly generate $\left(\V_c^{\otimes 2}\right)^{S_2}$, our first result minimizes this generating set.

\begin{thm} \label{mainrank2} 
For all $c\neq \frac{256}{47}$ the orbifold $\left(\V_c^{\otimes 2}\right)^{S_2}$  is strongly generated by $\omega$,  $w_4$, $w_6$, and $w_8$. Further $(\V_{\frac{256}{47}}^{\otimes 2})^{S_2}$ is strongly generated by $\omega$, $w_4$, $w_6$, $w_8$, and $w_{10}$.
\end{thm}
\begin{proof}
We diagonalize the action of the generator $(12)$ of $S_2$ by considering the change of basis
\begin{align}
\omega &= \omega_1 + \omega_2\\
u &= \omega_1 - \omega_2.
\end{align}
We also define the fields associated to these generators:
\begin{align}
L(z)&=Y(\omega,z) \\
U(z)&=Y(u,z)
\end{align}
When no confusion arises, we shall omit the $z$-variable and just write $L$ and $U$ for these fields, etc. Finally, we define the field
\begin{align}
W(a,b) = \nop{\partial^a U  \partial^b U}
\end{align}
We note here that the Leibniz rule holds for normal orderings:
\begin{align}\label{prodrule}
\partial W(a,b) = W(a+1,b) + W(a,b+1)
\end{align}
so that, with repeated applications of (\ref{prodrule}), we have
\begin{align}\label{Wab-rewrite}
W(a,b) = (-1)^b W(a+b,0) + \Psi
\end{align}
where $\Psi$ is a normally ordered polynomial of lower weight fields and their derivatives. We note here that $W(a,0)$ is a scalar multiple of $Y(u_{-a-1}u,z)$ for $a \ge 0$.

First, we note that, using (\ref{Wab-rewrite}), we have an initial strong generating set with fields $L$ and $W(a,0)$ where $a \ge 0$. In fact, using the derivation $\partial$ this set can be further reduced to contain only $L$ and $W(2a,0)$ where $a \ge 0$. Now observe that reduction to our desired generating set is can be achieved by construction of explicit relations that allow us to write $W(2a,0)$ for $a\geq 5$ ($a\geq 6$ for $c= \frac{256}{47}$) as a normally ordered polynomial in $L$ and $W(2b,0)$ for $0\leq b\leq 2$ ($0\leq b\leq 3$ for $c=\frac{256}{47}$).
We consider the following terms of weight $9+a$ for $a \ge 2$:
\begin{align}
C_1(a) &= \nop{W(a,0) W(1,0)} - \nop{W(a,1) W(0,0)}\\
C_2(a) &= \nop{W(a-1,1) W(1,0)} - \nop{W(a-1,0) W(1,1)}
\end{align}
Direct computation when $a$ is an even integer gives
\be\begin{aligned}
C_1(a) &= \frac{\left(a^4 (3 c+11)+a^3 (31 c+47)+51 a^2 (2 c-1)+a (134 c-137)+60 c+70\right)}{30 (a+1) (a+2) (a+4) (a+5)}W(a+5,0)+\\
&+\frac{\left(a^3 (c+14)+12 a^2 (c+6)+a (35 c+22)+24 c-12\right) }{24 (a+1) (a+3) (a+4)}W(a+4,1)+\\
&-\frac{\left(a^2+9 a+10\right) }{12 (a+1) (a+2)}W(a+2,3) + \frac{1}{6} \nop{(\partial^3 L)W(a,0) } + \frac{(a+2) }{a+1}\nop{(\partial^2 L)W(a+1,0)}+\\
&+\frac{\left(5 a^2+17 a+8\right) }{2 (a+1) (a+2)}\nop{(\partial L)W(a+2,0)} + \frac{(a+2)}{a+1} \nop{(\partial L)W(a+1,1)}+\\
& + \frac{2 \left(a^2+4 a+2\right) }{(a+2) (a+3)}\nop{LW(a+3,0)}+\frac{(a+4) }{a+2}\nop{LW(a+2,1)}
\end{aligned}\ee
and 
 \be\begin{aligned}
C_2(a) &=- \frac{ \left(a^4 (5 c+22)+3 a^3 (25 c+52)+4 a^2 (70 c+11)+12 a (25 c-37)+72\right) }{180a (a+2) (a+3) (a+4)}W(a+5,0)+\\
&-\frac{ \left(a^3 (6 c+19)+a^2 (38 c+37)+a (84 c-19)+52 c+323\right) }{60(a+1) (a+3) (a+4)}W(a+4,1)+\\
&-\frac{5 (a-1) }{24(a+3)}W(a+3,2)+\frac{ \left(2 a^2+6 a+7\right) }{9(a+1) (a+2)}W(a+2,3) -\frac{1}{360}W(a-1,6)-\frac{ (a-1) }{60a}W(a,5)+\\
&-\frac{1}{3} \nop{(\partial^3 L)W(a-1,1)}-\frac{ (a+2) }{2a}\nop{(\partial^2 L)W(a+1,0)} - \frac{2 (a-1)}{a} \nop{(\partial^2 L)W(a,1)}+\\
&-\frac{ (a+4) }{a+1}\nop{(\partial L)W(a+2,0)}-\frac{ \left(7 a^2-3 a+2\right) }{2a (a+1)}\nop{(\partial L)W(a+1,1)}-\frac{ (a+6) }{2(a+2)}\nop{LW(a+3,0)}+\\
&-\frac{2 \left(a^2+2 a+2\right) }{(a+1) (a+2)}\nop{LW(a+2,1)}
\end{aligned}
\ee 
Applying (\ref{Wab-rewrite}) to each of these yields: 
\be\begin{aligned}
C_1(a) =  \frac{p_1(a,c) }{120 (a+1) (a+2) (a+3) (a+4) (a+5)} W(a+5,0) + \Psi_1
\end{aligned}\ee
and
\be\begin{aligned}
C_2(a) = \frac{p_2(a,c) }{180 a (a+1) (a+2) (a+3) (a+4)}W(a+5,0) + \Psi_2
\end{aligned}\ee
where $\Psi_1$ and $\Psi_2$ are normally ordered polynomials of lower weight fields and their derivatives and
\be\begin{aligned}
p_1(a,c) &= a^5 (7 c-16)+5 a^4 (13 c-64)+15 a^3 (9 c-88)+a^2 (560-185 c)-106 a (7 c-76)+\\
&-480 c+7440\\
p_2(a,c) &= a^5 (13 c-40)+a^4 (70 c-556)+25 a^3 (5 c-68)+40 a^2 (2 c-23)+12 a (c+70)-144.
\end{aligned}\ee  
Noting that these the coefficient of $W(a+5,0)$ cannot be simultaneously $0$, we may sue these expressions to may replace $W(a+5,0)$ with a polynomial in of lower weight fields and their derivatives regardless of central charge $c$. 

When $a$ is odd, a similar argument gives us:
\be\begin{aligned}
C_1(a) = \frac{(a+9) \left(a^2 (7 c-16)+6 a (9 c-32)+80 (c-7)\right) }{120 (a+2) (a+4) (a+5)}W(a+5,0)+\Psi_1
\end{aligned}\ee
and 
\be\begin{aligned}
C_2(a)=\frac{(a+9) \left(a^3 (13 c-40)+12 a^2 (5 c-32)+68 a (c-10)+336\right) }{180 a (a+2) (a+3) (a+4)}W(a+5,0) + \Psi_2
\end{aligned}\ee
where $\Psi_1$ and $\Psi_2$ are normally ordered polynomials of lower weight fields and their derivatives,
and allows us to remove $W(a+5,0)$ from our list of generators for $a \ge 2$ regardless of central charge $c$.

At this point, the following generating fields remain: 
\begin{align}
L, W(0,0), W(2,0), W(4,0), W(6,0)
\end{align}
The lowest weight linear dependence relation among these operators occurs at weight $10$ and has the form:
\begin{align*}
0&=\frac{5}{16} (c-8) \partial^2 W(4,0)-\frac{1}{360} (47 c-256) W(6,0)-\frac{1}{48} (19 c-144) \partial^4 W(2,0)+\\
&-\frac{1}{480} (287-38 c) \partial^6 W(0,0)+\frac{11}{60} \nop{(\partial^6 L)L}-\frac{1}{4} \nop{(\partial^4L)W(0,0)}+\\
&-\frac{7}{6} \nop{(\partial^3 L)(\partial(W(0,0))}+\frac{3}{16} \nop{(\partial^3L)(\partial^3L)}-\frac{9}{2} \nop{(\partial^2L)W(2,0)}+\\
&-\frac{7}{2} \nop{(\partial L)(\partial W(2,0))}+\frac{1}{6} \nop{(\partial L)(\partial^3 W(0,0))}-\frac{1}{15} \nop{(\partial^5L)(\partial L)}+\\
&+\frac{7}{2} \nop{L (\partial^2W(2,0))}-\frac{5}{2}  \nop{LW(4,0)}-\frac{11}{12} \nop{L (\partial^4W(0,0))}+\\
&+\frac{1}{4} \nop{(\partial W(0,0))(\partial W(0,0))}-\frac{1}{2} \nop{(\partial^2 W(0,0))W(0,0)}+\nop{W(0,0)W(2,0)}+\\
&-\frac{(323 c-2838) }{20160}\partial^8 L
\end{align*}
and thus, we have that $W(6,0)$ can be removed from our list of generating fields unless $c = \frac{256}{47}$.
\end{proof}

For generic central charge, we can replace the above weight 4 and 6 generators with the following primary vectors (in addition to the conformal vector $\omega$).
\be\begin{aligned}\label{wt46}
\widehat{w}_4&=4(5c+11)J_4(-4)\1-2(5c+22)L(-2)^2\1-6cL(-4)\1\\
\widehat{w}_6&=108 (c + 12) (4 c - 1) (7 c + 34) J_6(-6)\1-12(4 c - 1) (5 c + 44) (7 c + 34) J_4(-6)\1\\&+24(5 c + 22) (32 c^2 + 345 c - 68)L(-6)\1-576 (4 c - 1) (7 c + 34) L(-2)J_4(-4)\1\\&+12 (2 c - 1) (5 c + 22) (7 c + 68)L(-3)^2\1-96 (c + 8) (2 c - 1) (7 c + 68)L(-4)L(-2)\1\\&+1152 (7 c^2 + 43 c - 17) L(-2)^3\1
\end{aligned}\ee
There are two primaries that involve $w_8$ that can be used as generators. After clearing denominators the coefficients of $w_8$ in each of these generators are such that their greatest common divisor is $(3c+23)(10c+3)$ and as such it is possible to take a linear combination of these vectors so that we have 
\be\label{wt8} 
\widehat{w}_8=(3c+23)(10c+3)J_8(-8)\1+\cdots.\ee
Further, for $c=\frac{256}{47}$ a primary can be similarly calculated at weight 10. These calculations, \eqref{wt46} and \eqref{wt8}, lead to the following 

\begin{cor} \label{cor-rank2} We have
\begin{enumerate}
\item For $c\not\in\{-12,-\frac{23}{3},-\frac{34}{7},-\frac{11}{5},-\frac{3}{10},\frac{256}{47}\}$ the orbifold $\left(\V_c^{\otimes 2}\right)^{S_2}$ is strongly generated by primary vectors of weight 2,4,6,8 and is thus of type (2,4,6,8). 
\item The orbifold $(\V_{\frac{256}{47}}^{\otimes 2})^{S_2}$ is strongly generated by primary vectors of weight 2,4,6,8,10 and is thus of type (2,4,6,8,10).

\item In all other cases $\left(\V_c^{\otimes 2}\right)^{S_2}$ is strongly generated by vectors in weight 2,4,6,8 some of which are not primary, as apparent from the coefficients in \eqref{wt46} and \eqref{wt8}.
\end{enumerate}

\end{cor}

%We can maybe use \verb+https://arxiv.org/abs/1805.11031+ to describe these...
\section{The simple orbifold $\left(\L_c^{\otimes 2}\right)^{S_2}$}

In order to describe the simple orbifold $\left(\L_c^{\otimes 2}\right)^{S_2}$, we first have to determine all simple vertex algebras $\L_c$ that have singular vectors up to and including degree $8$. Such vertex algebras 
are characterized by central charges:  $c_{2,5}=-\frac{22}{5}$, $c_{2,7}=-\frac{68}{7}$, $c_{2,9}=-\frac{46}{3}$, $c_{3,4}=\frac12$, and $c_{3,5}=-\frac35$.
Then using Theorem \ref{mainrank2} we get the following result. 
\begin{thm} \label{rank2simple} We have:
\begin{itemize}
\item[(a)]  $(\L_{-22/5} \otimes \L_{-22/5})^{S_2} \cong \L_{-44/5}$,
\item[(b)] $(\L_{-68/7} \otimes \L_{-68/7})^{S_2}$ is of type $(2,4)$,
\item[(c)] $(\L_{-3/5} \otimes \L_{-3/5})^{S_2}$ and $(\L_{-46/3} \otimes \L_{-46/3})^{S_2}$ are of type $(2,4,6)$ 
\item[(d)]   $(\L_{1/2} \otimes \L_{1/2})^{S_2}$ is of type $(2,4,8)$,
\item[(e)] In all other cases $(\L_c \otimes \L_c)^{S_2}$ is of the type described in Corollary \ref{cor-rank2}.
\end{itemize}
\end{thm}
\begin{proof}
Proofs are straightforward verifications with generators. We give a proof for part (c) only. 
It is known that $\V_{-3/5}$ and $\V_{-46/3}$ have a singular vector of degree 8 (also generator of the maximal ideal).
In these two cases the weight 8 generator, $w_8$ from Theorem  \ref{mainrank2}, can be written in terms of $\omega$, $w_4$, $w_6$, and the sum of the singular vectors of weight 8 from each $\L_{-3/5}$ (or $\L_{-46/3}$) component.
\end{proof}

\begin{rem}
Part (d) of the previous theorem can be approached using the fermionic construction of $c=\frac12$ minimal models.
Let $\F$ denote the rank one free fermion vertex superalgebra. Then  $\F^{\mathbb{Z}_2} \cong \L_{1/2}$ (the even subalgebra) and thus
\begin{align*}(\L_{1/2}\otimes \L_{1/2})^{S_2}\cong\left(\mathcal{F}^{\mathbb{Z}_2}\otimes \mathcal{F}^{\mathbb{Z}_2}\right)^{S_2} \cong (\mathcal{F}\otimes\mathcal{F})^G\end{align*}
where 
%$$G=\left<\begin{pmatrix}-1&0\\0&1\end{pmatrix},\begin{pmatrix}1&0\\0&-1\end{pmatrix},\begin{pmatrix}0&1\%\1&0\end{pmatrix}\right>,$$
$G$ viewed as a subgroup $\text{Aut}(\mathcal{F}\otimes \mathcal{F})=\mathcal{O}(2)$ is isomorphic to $D_4$. Using methods similar to those in \cite{MPW}, we have  $(\L_{1/2}\otimes \L_{1/2})^{S_2}\cong(\mathcal{F}\otimes \mathcal{F})^{D_4}$ is of type $(2,4,8)$. 

Alternatively, using results of Dong and Jiang on characterization of rational $c=1$ VOAs \cite{DJ} and known decomposition
$$(\L_{1/2} \otimes \L_{1/2})^{S_2} \cong {\bigoplus_{m \geq 0;even} L_{}(1,m^2)} \bigoplus {\bigoplus_{m \geq 1; even} L_{}(1,2m^2)},$$
we immediately get 
\begin{align*}
(\L_{1/2}\otimes \L_{1/2})^{S_2} \cong V_{4\mathbb{Z}}^+,
\end{align*}
which is known to be of type $(2,4,8)$.
\end{rem}

\section{A description in terms of the universal two parameter even spin VOA and affine $W$-algebras}

In \cite{KL}, the universal two parameter algebra, $\mathcal{W}^{\text{ev}}(c,\lambda)$, of type $\mathcal{W}(2,4,6,\dots)$ was rigorously constructed. This algebra is strongly generated by infinitely many fields in weights $2,4,6,\dots$ and weakly generated by a primary weight 4 field which we denote by $W_{\infty}^4$. We consider this algebra with central charge $2c$ to correspond with the central charge of our orbifold. A normalization can be chosen for this field so that
\be\begin{aligned}\label{theirprod}
\left(W_{\infty}^4\right)_{(3)}W_{\infty}^4=&32\lambda W_{\infty}^4-\frac{128(49\lambda^2(2c-1)(2c-25)-1)}{63(2c-1)(2c+24)(4c-1)}\nop{LL}\\&-\frac{32(2c-4)(49\lambda^2(2c-1)(2c-25))}{441(2c-1)(2c+24)(4c-1)}\partial^2 L.\end{aligned}\ee
For reference, we recall the following result from \cite{KL}

\begin{thm}\label{curves}(\cite{KL} Theorem 8.1)Let $c_0,c_1,\lambda_0,\lambda_1\in\mathbb{C}$ and let
$$I_0=(c-c_0,\lambda-\lambda_0),\hspace{.5in}I_1=(c-c_1,\lambda-\lambda_1)$$
be the corresponding maximal ideals in $\mathbb{C}[c,\lambda]$. Let $\mathcal{W}_0$ and $\mathcal{W}_1$ be the simple quotients of $\mathcal{W}^{\text{ev}}(c,\lambda)/I_0$ and $\mathcal{W}^{\text{ev}}(c,\lambda)/I_1$. Then $\mathcal{W}_0\cong\mathcal{W}_1$ only in the following cases.
\begin{enumerate}
\item $c_0=c_1$ and $\lambda_0=\lambda_1$
\item $c_0=c_1\in\{-24,-\frac{22}{5},0,\frac{1}{2},1\}$ and no restriction on $\lambda_0$, $\lambda_1$,
\item $c_0=c_1=c$, where $c\neq 1,25$, and
$$\lambda+0=\pm\frac{1}{7\sqrt{(c-25)(c-1)}}=\pm\lambda_1,$$
\item $c_0=c_1=c$, where $c\neq 1,25$, and
$$\lambda+0=\pm\frac{\sqrt{c^2-172c+196}}{21(c-1)(5c+22)}=\pm\lambda_1,$$

\end{enumerate}

\end{thm}

The basic strategy to apply Theorem \ref{curves} is to determine the value of $\lambda$ for each one parameter family of $\mathcal{W}(2,4,\dots, N)$ as a curve $\lambda=\lambda(c)$, this is known as the truncation curve of the vertex algebra. Next, we look for intersection points between truncation curves, leading to coincidental isomorphisms.

Back to our orbifold, the weight four generator for $\mathcal{V}_c^{\otimes 2}$ given by Theorem \ref{mainrank2} can be replaced by the primary field
\be
\widetilde{W}^4=\mu W^4-\frac{5c+22}{2(5c+11)}\mu\nop{LL}-\frac{3c}{4(5c+11)}\partial^2L,\ee
where $\mu$ we leave as a free normalization constant at the moment. It is also straightforward to check that this field weakly generates $\mathcal{V}_c^{\otimes 2}$. We have the following

\be\begin{aligned}\label{ourprod}
\left(\widetilde{W}^4\right)_{(3)}\widetilde{W}^4=&\frac{2(5c^2+33c-44)}{5c+11}\mu \widetilde{W}^4+\frac{21c(5c+22)}{(5c+11)^2}\mu^2\nop{LL}\\&+\frac{3c(c-2)(5c+22)}{2(5c+11)^2}\mu^2\partial^2L.\end{aligned}\ee

Now, equating \eqref{theirprod} and \eqref{ourprod} to solve for $\mu$ and $\lambda$ gives
\be
\mu=-\frac{16}{7(2c-1)(c+20)} \text{ and }\lambda=-\frac{5c^2+33c-44}{7(5c+11)(2c-1)(c+20)}.\ee

From this it follows that for $c\notin\{-20,-12,-\tfrac{11}{5},\frac{1}{2},\frac{1}{4}\}$ is a quotient of $\mathcal{W}^{\text{ev}}(c,\lambda)$ with truncation curve given by 
\be\label{curve1}
\lambda(c)=-\frac{5c^2+33c-44}{7(5c+11)(2c-1)(c+20)}.\ee

Another known quotient of the universal algebra $\mathcal{W}^{\text{ev}}(c,\lambda)$ is $\mathcal{W}_k(\mathfrak{so}_{2n},f_{\text{princ}})^{\mathbb{Z}_2}$ for $n\geq 3$\cite{KL}, whose truncation curve is given by 
\be\label{curve2}
\lambda(c)=\frac{38cn^3+24n^3+2c^2n^2-147cn^2+10n^2-2c^2n+120cn-28n-9c^2-6c}{7(c-1)(5c+22)(12n^3+2cn^2-19n^2-2cn+10n-3c)}\ee
where 
\be c=-\frac{n(2nk-2k+4n^2-10n+5)(2nk-k+4n^2-8n+4)}{k+2n-2},\ee
The category in which these algebras are objects has unique simple objects, so intersections of the curves \eqref{curve1} and \eqref{curve2} gives the isomorphisms
\be
(\L_{c^{\prime}}\otimes \L_{c^{\prime}})^{S_2}\cong \mathcal{W}_{k^{\prime}}(\mathfrak{so}_{2n},f_{\text{princ}})^{\mathbb{Z}_2},\ee
where 
\be
c^{\prime}=-\frac{n(4n-5)}{n+1}  \text{ and }k^{\prime}=-\frac{4n^2-2n-3}{2n+2},\ee
or
\be
c^{\prime}=-\frac{n(4n-5)}{n+1}  \text{ and }k^{\prime}=-\frac{4n(n-2)}{2n-1}.\ee
\begin{cor} \label{so} For $n\geq 3$, the vertex algebra $\mathcal{W}_{k^{\prime}}(\mathfrak{so}_{2n},f_{\text{princ}})^{\mathbb{Z}_2}$ is rational.
\end{cor}

A similar analysis can be done with the $\mathbb{Z}_2$ orbifold of the level $k$ $\mathfrak{sl}_2$ parafermion algebra, $N^{k}(\mathfrak{sl_2})^{\ZZ_2}$ which has central charge given by 
\be
c=\frac{3k}{k+2}-1=\frac{2(k-1)}{k+2}
\ee
and truncation curve given by 
\be 
\lambda(k)=\frac{(k+2)(4k^3-15k^2-33k-4)}{7(k-4)(16k+17)(k^2+3k+4)}.\ee
Intersections of the appropriate truncation curves give
\begin{cor} \label{KL}
\be
(\L_{7/10}\otimes \L_{7/10})^{S_2}\cong N_{8}(\mathfrak{sl}_2)^{\ZZ_2},\ee
where $N_k(\mathfrak{sl}_2)$ denotes the simple parafermionic algebra of level $k$.
\end{cor}
For further results regarding the $\mathbb{Z}_2$-orbifold of $N_{k}(\mathfrak{sl}_2),k \in \mathbb{N}$, see \cite{JW}.
\begin{rem} \label{drazen} Concerning Corollary \ref{KL}, there is another interesting pair of irrational 
vertex algebras $(\L_{-2}\otimes \L_{-2})^{S_2}$ and $N_{-1}(\mathfrak{sl}_2)^{\ZZ_2}$. These vertex algebras 
are both of type $(2,4,6,8)$ and central charge $c=-4$. Moreover, they are the only examples of vertex algebra orbifolds from each series that are {\em not} (weakly) generated by the primary vector of weight $4$. Instead the weight $4$ primary generates 
a subalgebra of type $(2,4)$. It would be interesting to investigate possible relationship between the two orbifolds.
\end{rem}

We can do the same type of calculation with $\mathcal{W}_k(\mathfrak{sp}(2m))$. In this case the truncation curves are quite long so we do not include them here, they can be found in \cite{KL}. 
\begin{cor} \label{sp}
We have $(\L_c\otimes \L_c)^{S_2}\cong \mathcal{W}_k(\mathfrak{sp}(2m))$ for the following values of $c$ and $k$
$$
c=-\frac{12m^2+10m}{2m+3}=c_{2,2m+3}\text{ and }k=-\frac{4m^2+8m+5}{4m+6},
$$
$$c=-\frac{3m^2-m-2}{m+2}=c_{2,m+2} \text{ and }k=-\frac{2m^2+m-2}{2m}$$
$$c=\frac{2m^2-5m}{2m^2-5m+3}=c_{2m-3,2m-2}\text{ and }k=-\frac{2m^2-2m-2}{2m-3},$$
where by $c_{p,q}$ we denote the appropriate minimal model. As such, we have established the rationality of these $\mathcal{W}$-algebras.
\end{cor}

\begin{rem} We should point out that three series of rational affine vertex algebras in Corollary \ref{sp} 
correspond to $(p,q)=(2m+1,4m+6)$, $(m+2,2m)$ and $(m-1,2m-3)$ models, that is the level is 
given by $-h^\vee +\frac{p}{q}$, where $h^\vee=m+1$ is the dual coxeter number of $\frak{sp}(2m)$. Using discussion in Section 9, it follows that the second and third series (for $m \geq 3$) is not 
among admissible series, so this gives a family of rational and lisse affine $W$-algebras outside
the main series.  Moreover, using Section 2, it is easy to write down explicit characters of these affine $W$-algebras. 
\end{rem}
Finally, analysis of two exceptional cases from Theorem \ref{rank2simple} (b,c), using similar methods, gives another pair of isomorphisms.
\begin{prop} We have
\begin{align*}
(\L_{-68/7} \otimes \L_{-68/7})^{S_2} & \cong W_{-\frac{16}{7}}(\frak{so}(5)) \\
(\L_{-46/3}\otimes \L_{-46/3})^{S_2} & \cong W_{-65/18}(\frak{sp}(6))\cong W_{-103/14}(\frak{sp}(14)).
\end{align*}
\end{prop}

\section{The orbifold  $\left(\V_c^{\otimes 3}\right)^{\mathbb{Z}_3}$}
In this section we define an alternative generating set for $\V_c^{\otimes 3}$ given by:
\begin{align}
\omega &= \omega_1 +\omega_2 +  \omega_3\\
u_1 &= \omega_1 + \eta^2 \omega_2 + \eta \omega_3\\
u_2 &= \omega_1 + \eta \omega_2 + \eta^2 \omega_3\\
\end{align}
where $\eta = e^{\frac{2 i \pi}{3}}$. We also define the fields associated to these generators by:
\be\begin{aligned}\label{changevars}
L(z)&= Y(\omega,z)\\
U_1(z) &= Y(u_1,z)\\
U_2(z)&= Y(u_2,z)
\end{aligned}\ee
When no confusion arises, we use the notation $L, U_1$, and $U_2$ for these fields, respectively.
Finally, we consider the following fields:
\be\begin{aligned}\label{origingens}
W_2(a,b) &= \nop{\partial^{a} U_1 \partial^b U_2}\\
W_2'(a,b) &= \nop{\partial^{a}{U_2} \partial^{b}U_1}\\ 
W_{3,1}(a,b,c) &= \nop{\partial^a U_1\partial^b U_1\partial^c U_1} \\
W_{3,2}(a,b,c) &= \nop{\partial^a U_2\partial^b U_2\partial^c U_2}
\end{aligned}\ee

We note here that the Leibniz rule holds for normal orderings:
\begin{equation}\label{prodrule1}
\partial W_2(a,b) = W_2(a+1,b) + W_2(a,b+1),
\end{equation}
so that, with repeated applications of (\ref{prodrule1}), we may always write
\begin{equation}\label{omega2rewrite}
W_2(a,b) = (-1)^b W_2(a+b,0) + \Psi
\end{equation}
where $\Psi$ is a normally ordered polynomial of lower weight fields and their derivatives. Thus, via direct calculation, we have the following useful facts 
\begin{align*}
W_{3,1}(a,b,0)  = W_{3,1}(b,a,0) &+ \frac{a! b! (a-b) }{(a+b+3) (a+b+4) (a+b+1)!} \partial^{a+b+4}L\\
& -\frac{(-1)^{a+b+2} a! b! (a-b) }{(a+b+2)!}W_2(a+b+2,0) + \Psi
\end{align*}
\begin{equation}
W_{3,2}(a,b,0) = W_{3,2}(b,a,0)  -\frac{a! b! (a-b) }{(a+b+2)!}W_2(a+b+2,0) + \Psi
\end{equation}
\begin{equation}\label{w31fix}
W_{3,1}(a,b,c) = W_{3,1}(a,c,b) +\frac{c! (-1)^{b+c+1} b! (b-c) }{(b+c+2)!}W_2(a+b+c+2,0) + \Psi
\end{equation}
and
\begin{align}\label{w32fix}
W_{3,2}(a,b,c) = W_{3,2}(a,c,b)  & +\frac{(-1)^{a+1} c! b! (b-c) }{(b+c+2)!}W_2(a+b+c+2,0) \\
& +\frac{a! b! c! (b-c) (a-b-c-2) }{(a+b+c+4)!}\partial^{a+b+c+4} L+ \Psi
\end{align}
where, in each case, $\Psi$ is some normally ordered polynomial of lower weight fields and their derivatives. 

We also recall that, using the Leibniz rule for normal orderings, we have:
\begin{equation}\label{w31productrule}
\partial W_{3,1}(a,b,c) = W_{3,1}(a+1,b,c) +W_{3,1}(a,b+1,c) +W_{3,1}(a,b,c+1) 
\end{equation}
so that any of the summands on the right hand side can be expressed in terms of the other two along with a derivative of a lower weight term. Applying (\ref{w31productrule}) repeatedly, we may write 
\begin{align*}
W_{3,1}(a,b,c) =(-1)^b \sum_{k=0}^b{b\choose k} W_{3,1}(a+k,0,c+b-k) + \Psi
\end{align*}
where $\Psi$ is a normally ordered polynomial of lower weight fields and their derivatives. In particular, setting $c=0$, we have that 
\begin{align*}
W_{3,1}(a,b,0) =(-1)^b \sum_{k=0}^b{b\choose k}W_{3,1}(a+k,0,b-k) + \Psi
\end{align*}
where $\Psi$ is a normally ordered polynomial of lower weight fields and their derivatives. We apply equation (\ref{w31fix}) to obtain
\begin{equation}\begin{aligned}
W_{3,1}(a,b,0) &= (-1)^b \sum_{k=0}^b{b\choose k} W_{3,1}(a+k,b-k,0) + \\
&+(-1)^b\sum_{k=0}^{b-1}{b\choose k} \frac{(-1)^{b-k+1}(b-k)!(k-b)}{(b-k+2)!} W_2(a+b+2,0)   +\Psi
\end{aligned}\end{equation}
In particular, when $b$ is odd, we may rewrite this equation as
\begin{align*}
W_{3,1}(a,b,0) &= \frac{1}{2}(-1)^b \sum_{k=1}^b{b\choose k} W_{3,1}(a+k,b-k,0) + \\
&+\frac{1}{2}(-1)^b\sum_{k=0}^{b-1}{b\choose k} \frac{(-1)^{b-k+1}(b-k)!(k-b)}{(b-k+2)!} W_2(a+b+2,0)     +\frac{ 1}{2}\Psi 
\end{align*}

In our work, we only need to consider the cases when $b=1,3,5$. In particular, repeated applications of the above yield (modulo derivatives of lower weight terms):
\begin{equation}\label{w31a1}
W_{3,1}(a,1,0) = -\frac{1}{2}W_{3,1}(a+1,0,0) +\frac{1}{12}W_2(a+3,0) + \Psi
\end{equation}
\begin{equation}
W_{3,1}(a,3,0) =-\frac{1}{20} W_{2}(a+5,0)-\frac{3}{2} W_{3,1}(a+1,2,0)+\frac{1}{4} W_{3,1}(a+3,0,0) + \Psi
\end{equation}
and 
\begin{equation}
W_{3,1}(a,5,0) =\frac{17}{168}W_{2}(a+7,0)-\frac{5}{2} W_{3,1}(a+1,4,0)+\frac{5}{2} W_{3,1}(a+3,2,0)-\frac{1}{2} W_{3,1}(a+5,0,0) + \Psi
\end{equation}
where in each case $\Psi$ is some normally ordered polynomial of lower weight fields and their derivatives.
Repeating these computations for $W_{3,2}$ yields:
\begin{equation}
W_{3,2}(a,1,0) = \frac{1}{12} (-1)^{a+1} W_{2}(a+3,0)-\frac{1}{2} W_{3,2}(a+1,0,0) + \Psi
\end{equation}
\begin{equation}
W_{3,2}(a,3,0) = \frac{1}{20} (-1)^a W_{2}(a+5,0)-\frac{3}{2} W_{3,2}(a+1,2,0)+\frac{1}{4} W_{3,2}(a+3,0,0)+ \Psi
\end{equation}
and 
\begin{equation}\label{w32a5}\begin{aligned}
W_{3,2}(a,5,0) &=\frac{-17}{168} (-1)^aW_{2}(a+7,0)-\frac{5}{2} W_{3,2}(a+1,4,0)+\\
&+\frac{5}{2} W_{3,2}(a+3,2,0)-\frac{1}{2} W_{3,2}(a+5,0,0) + \Psi
\end{aligned}\end{equation}
where in each case $\Psi$ is some normally ordered polynomial of lower weight fields and their derivatives.
It is clear that (\ref{w31a1}) - (\ref{w32a5}) can be extended to rewrite any relation of the form $W_{3,i}(a,b,0)$, where $b$ is odd, in terms of $W_{3,i}(a+1,b-1,0), W_{3,i}(a+3,b-3,0), \dots ,W_{3,i}(a+b,0,0)$ and a normally ordered polynomial of lower weight fields and their derivatives for $i=1,2$.

\begin{thm}\label{Z3genreduce}
The orbifiold $\left(\V_c^{\otimes 3}\right)^{\mathbb{Z}_3}$ is of type $(2, 4, 5, 6^3 , 7, 8^3 , 9^3 , 10^2 )$ when 
$c \neq 0$, $ -\frac{9}{16}$, $ \frac{9}{352} \left(3661\pm\sqrt{12376489}\right)$.
Moreover, when $c =  -\frac{9}{16}, \frac{9}{352} \left(3661\pm\sqrt{12376489}\right)$, then the orbifold $\left(\V_c^{\otimes 3}\right)^{\mathbb{Z}_3}$ is of type $(2, 4, 5, 6^3 , 7, 8^3 , 9^3 , 10^2,11 )$. 
\end{thm}
\begin{proof}
First, using an argument identical to \cite{MPSh}, we have an initial strong generating set of fields given by:
\begin{equation}
\{ L, W_2(a,0), W_{3,1}(b,d,0), W_{3,2}(b,d,0) | a \ge 0, b \ge d \ge 0 \}
\end{equation}

First, we being by defining, for $a,b \ge 0$, the terms:
\begin{align}
D_1(a,b) &= \nop{W_2(a,0) W_2(0,b)}- \nop{W_2(a,b) W_2(0,0)}\\
D_2(a,b) &= \nop{W_2(a,0) W_2(1,b)} - \nop{W_2(a,b) W_2(1,0)}
\end{align}
We additionally consider 
\begin{align}
D_1'(a,b) &= \nop{W'_2(a,0) W'_2(0,b)} - \nop{W'_2(a,b) W'_2(0,0)}\\
D_2'(a,b) &= \nop{W'_2(a,0) W'_2(1,b)} - \nop{W'_2(a,b) W'_2(1,0)}
\end{align}

We first show that $W_2(a+6,0)$, $W_{3,1}(a+4,0,0)$, $W_{3,1}(a+2,2,0)$, $W_{3,1}(a,4,0)$,  $W_{3,2}(a+4,0,0)$, $W_{3,2}(a+2,2,0)$, and $W_{3,2}(a,4,0)$ can be expressed using some normally ordered polynomial of lower weight fields and their derivatives whenever $a \ge 6$. First, for $a \ge 6$, we have the following via direct computation:
{ 
\begin{align*}
D_1(a+1,1) &= d_{1,1,}(a,c) W_2(a+6,0)+\frac{1}{6} W_{3,1}(a+1,3,0)+W_{3,1}(a+2,2,0)+\\
&+2 W_{3,1}(a+3,1,0)+\frac{4}{3} W_{3,1}(a+4,0,0)+\frac{(-1)^a }{a+2}W_{3,2}(a+2,2,0)+\\
&+\frac{2 (-1)^a }{(a+2) (a+3)}W_{3,2}(a+3,1,0)+\frac{(2 (-1)^a }{a+3} W_{3,2}(a+4,0,0)+ \Psi
\end{align*}

\begin{align*}
D'_1(a+1,1) &=(-1)^a d_{1,1}(a,c)W_2(a+6,0)+\frac{(-1)^a }{a+2}W_{3,1}(a+2,2,0)+\frac{2 (-1)^a }{(a+2) (a+3)}W_{3,1}(a+3,1,0)+\\
&+\frac{2 (-1)^a }{a+3}W_{3,1}(a+4,0,0)+\frac{1}{6} W_{3,2}(a+1,3,0)+W_{3,2}(a+2,2,0)+\\
&+2 W_{3,2}(a+3,1,0)+\frac{4}{3} W_{3,2}(a+4,0,0) + \Psi
\end{align*}

\begin{align*}
D_1(a,2)&=d_{1,2}(a,c) W_2(a+6,0)+\frac{1}{6}W_{3,1}(a,4,0)+W_{3,1}(a+1,3,0)+2W_{3,1}(a+2,2,0)+\\
&+\frac{4}{3} W_{3,1}(a+3,1,0)+\frac{(-1)^{a+1} }{a+1}W_{3,2}(a+1,3,0)+\frac{(-1)^a a }{(a+1) (a+2)}W_{3,2}(a+2,2,0)+\\
&+\frac{(-1)^{a+1} }{a+1}W_{3,2}(a+3,1,0)-\frac{2 (-1)^a }{a+2}W_{3,2}(a+4,0,0) + \Psi
\end{align*}

\begin{align*}
D'_1(a,2)&=(-1)^ad_{1,2}(a,c) W_2(a+6,0)+\frac{(-1)^{a+1} }{a+1}W_{3,1}(a+1,3,0)+\frac{(-1)^a a }{(a+1) (a+2)}W_{3,1}(a+2,2,0)+\\
&+\frac{(-1)^{a+1} }{a+1}W_{3,1}(a+3,1,0) -\frac{2 (-1)^a}{a+2}W_{3,1}(a+4,0,0)+\frac{1}{6} W_{3,2}(a,4,0)+\\
&+W_{3,2}(a+1,3,0)+2 W_{3,2}(a+2,2,0)+
\frac{4}{3} W_{3,2}(a+3,1,0) + \Psi
\end{align*}

\begin{align*}
D_1(a-1,3) &= d_{1,3}(a,c)W_2(a+6,0)+\frac{3}{20} W_{3,1}(a-1,5,0)+\frac{5}{2}W_{3,1}(a+1,3,0)+3 W_{3,1}(a+2,2,0)+\\
&+2W_{3,1}(a+3,1,0)+\frac{4}{5} W_{3,1}(a+4,0,0)+W_{3,1}(a,4,0)+\\
&+\frac{(-1)^{a+1} (a-1) }{a (a+1)}W_{3,2}(a+1,3,0)+\frac{(-1)^a }{a}W_{3,2}(a+3,1,0)+\\
&+\frac{2 (-1)^a }{a+1}W_{3,2}(a+4,0,0)+\frac{(-1)^a }{a} W_{3,2}(a,4,0)+ \Psi
\end{align*}

\begin{align*}
D'_1(a-1,3) &=(-1)^a d_{1,3}(a,c) W_2(a+6,0)+
+\frac{3}{20}W_{3,2}(a-1,5,0)+\frac{5}{2}W_{3,2}(a+1,3,0)+\\
&+3 W_{3,2}(a+2,2,0)+2W_{3,2}(a+3,1,0)+\frac{4}{5}W_{3,2}(a+4,0,0)+\\
&+W_{3,2}(a,4,0)+\frac{(-1)^{a+1} (a-1)}{a (a+1)}W_{3,1}(a+1,3,0)+\frac{(-1)^a }{a} W_{3,1}(a+3,1,0)\\ 
& +\frac{2 (-1)^a}{a+1}W_{3,1}(a+4,0,0)+\frac{(-1)^a}{a} W_{3,1}(a,4,0)+ \Psi
\end{align*}

\begin{align*}
D_2(a-1,2)&=d_{2,2}(a,c) W_2(a+6,0)+\\
&-\frac{1}{6}W_{3,1}(a-1,5,0)-3 W_{3,1}(a+1,3,0)-\frac{10}{3} W_{3,1}(a+2,2,0)-\frac{4}{3} W_{3,1}(a+3,1,0)+\\
&-\frac{7}{6} W_{3,1}(a,4,0)+\frac{(-1)^a (a-2) }{a (a+1)}W_{3,2}(a+1,3,0)+\frac{(-1)^a (a+4)}{(a+1) (a+2)}W_{3,2}(a+2,2,0)+\\
&+\frac{3 (-1)^a }{a+1}W_{3,2}(a+3,1,0)+\frac{2 (-1)^a }{a+2}W_{3,2}(a+4,0,0)+\frac{(-1)^{a+1} }{a}W_{3,2}(a,4,0)+ \Psi
\end{align*}
}
where in each case $\Psi$ is some normally ordered polynomial of lower weight fields and their derivatives and where 
\be\begin{aligned}
d_{1,1}(a,c) = -\frac{2 a^5+35 a^4+245 a^3+940 a^2+2093 a+2010}{15 (a+2) (a+3) (a+4) (a+5) (a+6)}-\frac{((a+1) (a+10)) c}{80 (a+5) (a+6)},
\end{aligned}\ee
\be \begin{aligned}
d_{1,2}(a,c) = \frac{(a (a+10)) c}{48 (a+4) (a+6)}-\frac{16 a^6+423 a^5+4120 a^4+17685 a^3+29704 a^2+3492 a-19440}{180 (a+1) (a+2) (a+3) (a+4) (a+5) (a+6)},
\end{aligned}\ee
\be\begin{aligned}
d_{1,3}(a,c) = -\frac{8 a^6-95 a^5-1937 a^4-4117 a^3+24105 a^2+67968 a+23940}{420 a (a+1) (a+2) (a+3) (a+5) (a+6)}-\frac{(3 (a-1) (a+10)) c}{112 (a+3) (a+6)},
\end{aligned}\ee
and
\be\begin{aligned}
d_{2,2}(a,c) &= \frac{(a-2) (a+10) c}{70 (a+3) (a+5)}\\
&-\frac{16 a^7+567 a^6+7280 a^5+42805 a^4+110544 a^3+65828 a^2-134640 a-100800}{420 a (a+1) (a+2) (a+3) (a+4) (a+5) (a+6)}
\end{aligned}\ee

Now, using (\ref{w31a1}) - (\ref{w32a5}) to rewrite the above expressions, we collect the coefficients of  $W_2(a+6,0)$, $W_{3,1}(a+4,0,0)$, $W_{3,1}(a+2,2,0)$, $W_{3,1}(a,4,0)$, $W_{3,2}(a+4,0,0)$, $W_{3,2}(a+2,2,0)$, and $W_{3,2}(a,4,0)$ to form a $7 \times 7$ matrix and solve for  $W_2(a+6,0)$, $W_{3,1}(a+4,0,0)$, $W_{3,1}(a+2,2,0)$, $W_{3,1}(a,4,0)$, $W_{3,2}(a+4,0,0)$, $W_{3,2}(a+2,2,0)$, and $W_{3,2}(a,4,0)$ in terms of $D_1(a+1,1)$, $D'_1(a+1,1)$, $D_1(a,2)$, $D'_1(a,2)$,  $D_1(a-1,3)$, $D'_1(a-1,3)$, $D_2(a-1,2)$, and some normally ordered polynomial of lower weight fields and their derivatives. The determinant of this matrix is:
\begin{equation}
-\frac{2187 (a-3) (a-2)^2 (a-1)^2 (a+10) \left(3 a^3+39 a^2+150 a+160\right) c}{229376000 a (a+2)^2 (a+3)^3 (a+4) (a+5) (a+6)}
\end{equation}
and we see that it is invertible precisely when $c \neq 0$.

By direct computation, we have
\begin{align*}
D_2(a,b) &= \sum _{k=0}^b \frac{\left((-1)^k-3\right) b! \left(-\left((-1)^k+3\right) \left(-2 b^2+b (k-5)+2 k\right)\right) }{4 (2 (k+3)!) (b-k)!}W_{3,1}(a+b-k,k+3,0)+\\
&+\frac{\left((-1)^{b+1} \left((-1)^b b^2+3 (-1)^b b+(-1)^b-1\right)\right) }{b+1}W_{3,1}(a+b+1,2,0)+\\
&+\frac{\left((-1)^{b+1} b \left(2 (-1)^b b+3 (-1)^b+1\right)\right)}{(b+1) (b+2)} W_{3,1}(a+b+2,1,0)+\\
&+\frac{\left((-1)^a a!\right) }{(a+1)!}W_{3,2}(a+1,b+2,0)+\frac{\left((-1)^{a+1} a!\right) }{(a+1)!}W_{3,2}(a+b+1,2,0)+\\
&+\left(\frac{2 (-1)^a a!}{(a+1)!}-\frac{3 (-1)^a (a+1)!}{(a+2)!}\right) W_{3,2}(a+2,b+1,0)+\\
&-\frac{\left(3 (-1)^a (a+1)!\right)}{(a+2)!} W_{3,2}(a+b+2,1,0)+\\
&+\left(\frac{(-1)^a a!}{(a+1)!}-\frac{3 (-1)^a (a+1)!}{(a+2)!}+\frac{2 (-1)^a (a+2)!}{(a+3)!}\right) W_{3,2}(a+3,b,0)+\\
&-\frac{\left(2 (-1)^a (a+2)!\right) }{(a+3)!}W_{3,2}(a+b+3,0,0) + \Psi
\end{align*}
for $a \ge 4$ and $2 \le b \le a-2$, where $\Psi$ is a normally ordered polynomial of lower weight fields and their derivatives, and $W_2(a+b+5,0)$.
Using this expression, we may rewrite $W_{3,1}(a,b+3,0,0)$ in terms of $W_{3,1}(a+b-k,k+3,0)$ with $-2 \le k \le b-1$, $W_{3,2}(a+1,b+2,0)$, $W_{3,2}(a+b+1,2,0)$, $W_{3,2}(a+2,b+1,0)$, $W_{3,2}(a+b+2,1,0)$, $W_{3,2}(a+3,b,0)$, $W_{3,2}(a+b+3,0,0)$, $D_2(a,b)$, a normally ordered polynomial of lower weight fields and their derivatives, and $W_2(a+b+5,0)$.
Similarly, we can directly compute $D_2'(a,b)$ and obtain that, 
for $a \ge 4$ and $2 \le b \le a-2$, $W_{3,2}(a,b+3,0,0)$ can be rewritten in terms of $W_{3,2}(a+b-k,k+3,0)$ with $-2 \le k \le b-1$, $W_{3,1}(a+1,b+2,0)$, $W_{3,1}(a+b+1,2,0)$, $W_{3,1}(a+2,b+1,0)$, $W_{3,1}(a+b+2,1,0)$, $W_{3,1}(a+3,b,0)$, $W_{3,1}(a+b+3,0,0)$, $D'_2(a,b)$, a normally ordered polynomial of lower weight fields and their derivatives, and $W_2(a+b+5,0)$.

Next, we eliminate $W_2(11,0)$, $W_{3,1}(5,4,0)$, $W_{3,1}(7,2,0)$, $W_{3,1}(9,0,0)$, $W_{3,2}(5,4,0)$, \\ $W_{3,2}(7,2,0)$, and $W_{3,2}(9,0,0)$. We use the following $8$ relations to set up two $7 \times 7$ matrices which allow us to remove $W_2(11,0)$, $W_{3,1}(5,4,0)$, $W_{3,1}(7,2,0)$, $W_{3,1}(9,0,0)$, $W_{3,2}(5,4,0)$, $W_{3,2}(7,2,0)$, and $W_{3,2}(9,0,0)$ regardless of $c$:
\begin{align*}
D_1(6,1) &= -\frac{9}{880} c W_2(11,0)+\frac{73 }{1540}W_2(11,0)+\frac{3}{4}W_{3,1}(7,2,0)+\frac{3}{8} W_{3,1}(9,0,0)+\\
&-\frac{1}{7} W_{3,2}(7,2,0)-\frac{13}{56} W_{3,2}(9,0,0) +\Psi
\end{align*}
\begin{align*}
D_1(5,2) &= \frac{25 c }{1584}W_2(11,0)-\frac{10477 }{166320}W_2(11,0)+\frac{1}{6} W_{3,1}(5,4,0)+\frac{1}{2} W_{3,1}(7,2,0)+\\
&-\frac{5}{12}W_{3,1}(9,0,0)-\frac{31}{84} W_{3,2}(7,2,0)+\frac{41}{168} W_{3,2}(9,0,0)+\Psi
\end{align*}
\begin{align*}
D_1(4,3)&=-\frac{45 c }{2464}W_2(11,0)+\frac{2371}{30800} W_2(11,0)+\frac{23}{20} W_{3,1}(5,4,0)-\frac{3}{4} W_{3,1}(7,2,0)+\\
&+\frac{17}{40} W_{3,1}(9,0,0)-\frac{1}{5} W_{3,2}(5,4,0)-\frac{1}{5} W_{3,2}(7,2,0)-\frac{1}{5} W_{3,2}(9,0,0)+\Psi
\end{align*}
\begin{align*}
D_1'(6,1) &= \frac{9}{880} c W_2(11,0)-\frac{73 }{1540}W_2(11,0)-\frac{1}{7} W_{3,1}(7,2,0)+\\
&-\frac{13}{56} W_{3,1}(9,0,0)+\frac{3}{4} W_{3,2}(7,2,0)+\frac{3}{8} W_{3,2}(9,0,0)+\Psi
\end{align*}
\begin{align*}
D_1'(5,2) &= -\frac{25 c }{1584}W_2(11,0)+\frac{10477 }{166320}W_2(11,0)-\frac{31}{84} W_{3,1}(7,2,0)+\\
&+\frac{41}{168} W_{3,1}(9,0,0)+\frac{1}{6} W_{3,2}(5,4,0)+\frac{1}{2} W_{3,2}(7,2,0)-\frac{5}{12} W_{3,2}(9,0,0)+\Psi
\end{align*}
\begin{align*}
D_1'(4,3) &= \frac{45 c }{2464}W_2(11,0)-\frac{2371}{30800}W_2(11,0)-\frac{1}{5}W_{3,1}(5,4,0)-\frac{1}{5}W_{3,1}(7,2,0)+\\
&-\frac{1}{5} W_{3,1}(9,0,0)+\frac{23}{20}W_{3,2}(5,4,0)-\frac{3}{4} W_{3,2}(7,2,0)+\frac{17}{40} W_{3,2}(9,0,0)+\Psi
\end{align*}
\begin{align*}
D_2(4,2) &= \frac{9 c }{1120}W_2(11,0)-\frac{118 }{2475}W_2(11,0)-\frac{3}{4} W_{3,1}(5,4,0)+\frac{3}{4} W_{3,1}(7,2,0)+\\
&+\frac{1}{5} W_{3,2}(5,4,0)-\frac{9}{140} W_{3,2}(7,2,0)-\frac{17}{280} W_{3,2}(9,0,0)+\Psi
\end{align*}
\begin{align*}
D_2(3,3) &= -\frac{9 c }{1120}W_2(11,0)+\frac{12259}{92400}W_2(11,0)-\frac{23}{20} W_{3,1}(5,4,0)+\frac{39}{40} W_{3,1}(7,2,0)+\\
&-\frac{5}{16} W_{3,1}(9,0,0)+\frac{29}{40} W_{3,2}(5,4,0)-\frac{2}{5}W_{3,2}(7,2,0)+\frac{13}{80} W_{3,2}(9,0,0)+\Psi
\end{align*}
where in each case $\Psi$ is some normally ordered polynomial involving lower weight terms and their derivatives.
Now, extracting the coefficients of $W_2(11,0)$, $W_{3,1}(5,4,0)$, $W_{3,1}(7,2,0)$, $W_{3,1}(9,0,0)$, $W_{3,2}(5,4,0)$, $W_{3,2}(7,2,0)$, and $W_{3,2}(9,0,0)$ in $D_1(6,1)$ through $D_2(4,2)$, we form a $7 \times 7$ matrix whose determinant is 
\begin{align}
\frac{898794549203512500 c-7006017689331806250}{745253681356800000000}
\end{align}
and is invertible precisely when $c \neq \frac{177147}{22726}$, allowing us to rewrite $W_2(11,0)$, $W_{3,1}(5,4,0)$, $W_{3,1}(7,2,0)$, $W_{3,1}(9,0,0)$, $W_{3,2}(5,4,0)$, $W_{3,2}(7,2,0)$, and $W_{3,2}(9,0,0)$ in terms of $D_1(6,1)$ through $D_2(4,2)$ a normally ordered polynomial of lower weight fields and their derivatives when  $c \neq \frac{177147}{22726}$.
Similarly, replacing $D_2(4,2)$ in the above work with $D_2(3,3)$ yields a $7 \times 7$ matrix with determinant 
\begin{align}
\frac{672577698175853400000-86414789000769075000 c}{11924058901708800000000}
\end{align}
and is invertible precisely when $c \neq \frac{1417176}{182083}$. 
A similar strategy may now be employed to reduce our list of generating fields to
\begin{align*}
W_2(a,0) \ \  0 \le a \le 8\\
W_{3,1}(a,0,0) \ \  0 \le a \le 6\\
W_{3,2}(a,0,0) \ \  0 \le a \le 6\\
\end{align*}

%\color{red} Should we put the following full relations in an Appendix for weights 10, 11, and 12? They're quite large even %once we've sent L to 0. \color{black} 
Next, we establish linear dependence relations at different conformal weights to remove the remaining extraneous generating fields.
First, we note that there are no linear dependence relations among generators and their derivatives of weight up to and including 6.

At weight 7, we have the following linear dependence relations:
\begin{equation}
20 W_2(3,0)-40 \partial W_{3,2}(0,0,0)+120 W_{3,2}(1,0,0)+\partial^5 L = 0
\end{equation}
and
\begin{align}
-\frac{3}{2} \partial W_2(2,0)&+\frac{3}{2} \partial^2W_2(1,0)-\frac{1}{2} W_{3,2}(3,0)-\frac{1}{2} \partial^3W_2(0,0)+\\
&+\partial W_{3,1}(0,0,0)-3 W_{3,1}(1,0,0)+2\partial W_{3,2}(0,0,0)-6 W_{3,2}(1,0,0) = 0 
\end{align}
allowing us to rewrite $W_{3,1}(1,0,0)$ and $W_{3,2}(1,0,0)$ in terms of $W_2(3,0)$ and a normally ordered polynomial of lower weight fields and their derivatives.

At weight 8 and 9 there are no linear dependence relations among involving nonzero multiples of the generators and involving their derivatives.

At weight 10, we have relations involving $-\frac{1}{180} (c+30) W_2(6,0),  W_{3,1}(4,0,0), W_{3,2}(4,0,0)$ and multiples of their derivatives which allow us to rewrite one of the weight 10 generators in terms of the other 2.

At weight $11$, when $c \neq  -\frac{9}{16}, \frac{9}{352} \left(3661\pm\sqrt{12376489}\right)$, we may form three relations which allow us to rewrite $W_2(7,0)$, $W_{3,1}(5,0,0)$, and $W_{3,2}(5,0,0)$ as a linear combinations of derivatives of lower weight generators. When $c =  -\frac{9}{16}, \frac{9}{352} \left(3661\pm\sqrt{12376489}\right)$, we may rewrite 2 of the generators $W_2(7,0)$, $W_{3,1}(5,0,0)$, and $W_{3,2}(5,0,0)$ in terms of the third.

 At weight $12$, the generators $W_2(8,0), W_{3,1}(6,0,0), W_{3,2}(6,0,0)$ can be removed using $3$ linear dependence relations among $W_2(8,0), W_{3,1}(6,0,0), W_{3,2}(6,0,0)$ and derivatives of lower weight generators.
\end{proof}

%\section{The orbifold $\left(\L_c^{\otimes 3}\right)^{\mathbb{Z}_3}$}

\section{The orbifolds $\left(\L_{c}^{\otimes 3}\right)^{\mathbb{Z}_3}$  and $\left(\L_{-22/5}^{\otimes 3}\right)^{S_3}$} 

In this section with look at the $\mathbb{Z}_3$ and $S_3$ orbifolds of the three fold tensor product of the Virasoro vertex operator algebra at central charge $-\frac{22}{5}$. This is the simplest and by far the most beautiful nontrivial example as the singular vector occurs in weight four, allowing for significant simplification. The character of $\L_{-\frac{22}{5}}$ is 
given by the famous Rogers-Ramanujan product series:
$${\rm ch}[\L_{-\frac{22}{5}}](\tau)=\frac{q^{\frac{11}{60}}}{\prod_{i \geq 0}(1-q^{5i+2})(1-q^{5i+3})}.$$
 Our first main result examines the $\mathbb{Z}_3$ case.

\begin{thm}\label{c225} The vertex algebra $\left(\L_{-22/5}^{\otimes 3}\right)^{\mathbb{Z}_3}$ is of type $(2,5,6,9)$.
\end{thm}
\begin{proof}
We use the notation described in \eqref{changevars} and \eqref{origingens} and the well known singular vector
\be
v_k=10L_k(-2)^2-6L_k(-4)\1
\ee 
for $k=1,2,3$, in each copy of $\mathcal{V}_{-\frac{22}{5}}$, which is clearly 0 in the simple quotient $\L_{-22/5}$. As in previous proofs we perform our calculations at the level of vertex operators and so we set 
$$\begin{aligned}
V^s(z)&=Y(v_1+v_2+v_3,z)\\
V^s_1(z)&=Y(v_1+\eta^2 v_2+\eta v_3,z)\\
V^s_2(z)&=Y(v_1+\eta v_2+\eta^2 v_3,z)\\
\end{aligned}$$
and so under the generator of $\mathbb{Z}_3$ we have 
$$\begin{aligned}
V^s&\mapsto V^s\\
V^s_1&\mapsto \eta V^s_1\\
V^s_2&\mapsto \eta^2 V^s_2,\end{aligned}$$
which is parallel to the $\mathbb{Z}_3$ action on the generators $L$, $U_1$, and $U_2$ of $\left(\mathcal{V}_{-\frac{22}{5}}\right)^{\otimes 3}$. We now form the fields
\be\begin{split}
W^s_{2,1}(a)&=\nop{(\partial ^a U_2)V^s_1}\\
W^s_{2,2}(a)&=\nop{(\partial ^a U_1)V^s_2}\\
W^s_{2,3}(a)&=\nop{(\partial ^a V_1)V^s_2}\\
\end{split}
\hspace{.5in}
\begin{split}
W^s_{3,1}(a)&=\nop{(\partial^a U_1)U_1 V^s_1}\\
W^s_{3,1}(a)&=\nop{(\partial^a U_2)U_2 V^s_2}\\
\end{split}\ee
Observe that these fields are all descendants of the singular vectors and are thus zero in the simple quotient. Further the generators described in the proof of Theorem \ref{Z3genreduce} can be reduced down to the set containing the total conformal vector $\omega$ along with $w_k=U_1(-3)U_2(-2)$ for $k=5,6,9$, where $Y(u_i,z)=\sum_{n\in\mathbb{Z}}U_i(n)z^{-n-2}$. As such, we need to write the fields corresponding to the remaining generators in terms of $V^s$ and $W^s_{i,j}(a)$. For example, at weight four we have 
$$W_2(4,0)=\frac{3}{20}V^s-\frac{1}{2}\nop{LL}-\frac{9}{20}\partial^2L$$
inside the simple quotient. Similar equations exist for all remaining generators except those described above. For example, at weight 8 we require three such equations:
$$\begin{aligned}
W_{3,1}(2,0,0)&=-\frac{17}{56}W^s_{2,1}(2)-\frac{13}{168}W^s_{2,2}(2)+\frac{1}{84}W^s_{2,3}(0)-\frac{157}{758}W^s_{3,1}(0)-\frac{95}{756}W^s_{3,2}+\Psi_1,\\
W_{3,2}(2,0,0)&=-\frac{17}{56}W^s_{2,1}(2)-\frac{13}{168}W^s_{2,2}(2)+\frac{1}{84}W^s_{2,3}(0)+\frac{95}{756}W^s_{3,1}(0)-\frac{347}{756}W^s_{3,2}+\Psi_2,\\
W_{2}(4,0)&=\frac{9}{28}W^s_{2,1}(2)-\frac{1}{28}W^s_{2,2}(2)-\frac{1}{14}W^s_{2,3}(0)-\frac{25}{126}W^s_{3,1}(0)+\frac{25}{126}W^s_{3,2}+\Psi_3.\\
\end{aligned}$$

\end{proof}

\begin{rem}  A complete list of minimal central charges for which the type of $(\L_c^{\otimes 3})^{\mathbb{Z}_3}$ is
different than the generic type is necessarily among
 $c_{2,5},c_{3,4},c_{3,5},c_{2,7},c_{2,9}$, $c_{2,11}$. In these cases we 
get types: $(2,5,6,9)$, $(2,4,5,6^2,7,8,9)$, $(2,4,5,6^3,7,8^2,9^3)$, $(2,4,5,6^2,7,8,9^2)$, $(2,4,5,6^3,7,8^3,9^2,10^2)$, and $(2,4,5,6^3,7,8^3,9^3,10^2)$, respectively. 
As far as we know, there is no (universal) affine $W$-algebra of these types.
Except for $c_{2,5}$ it also interesting to look at $c_{3,4}=\frac12$ due to fermionic construction. In this 
case we get that $(\L_c^{\otimes 3})^{\mathbb{Z}_3}$ is isomorphic to $L_2(\frak{sl}(2))^{A_4 \times \mathbb{Z}_2}$.

We shall return to these examples in the sequel \cite{MPS2}.

\end{rem}

\begin{thm}
The vertex algebra $\left(\L_{-22/5}^{\otimes 3}\right)^{S_3}$ is of type $(2,6)$.
\end{thm}

\begin{proof}
Our approach will be to realize $\left(\L_{-22/5}^{\otimes 3}\right)^{S_3}$ as the $S_2\cong \mathbb{Z_2}$ orbifold of $\left(\L_{-22/5}^{\otimes 3}\right)^{\mathbb{Z}_3}$ where the automorphism permutes the vectors $u_1$ and $u_2$ as described in $\eqref{changevars}$. We make the important observation that at this central charge there is no need for ``cubic'' generators such at $\nop{U_1 U_1 U_1}+\nop{U_2 U_2 U_2}$, as the individual summands have been removed from the $\mathbb{Z}_3$ orbifold generating set. As such, we focus on the quadratic generators by considering a slightly different set of quadratic generators for the orbifold $\left(\V^{\otimes 3}_c\right)^{\mathbb{Z}_3}$ than those described in Theorem \ref{Z3genreduce}:
\be\begin{aligned}
X^+(a,b)&=\nop{(\partial^a U1)(\partial^b U2)}+\nop{(\partial^a U2)(\partial^b U1)}\\
X^-(a,b)&=\nop{(\partial^a U1)(\partial^b U2)}-\nop{(\partial^a U2)(\partial^b U1)}.
\end{aligned}\ee
Next in parallel with Theorem \ref{Z3genreduce} and Theorem \ref{c225} this infinite list of generators for the $\mathbb{Z}_3$ orbifold can be replaced with those fixed under the additional $\mathbb{Z}_2$ action: $X^+(2,0)$
a which is weight $6$, as well as those sent to their negative: $X^-(1,0) \text{ and } X^-(5,0),$ of weight 5 and 9. Next we see that by taking appropriate symmetric combinations of the relations given in the proof of Theorem \ref{Z3genreduce} all terms in that are quadratic in $\left(\L_{-22/5}^{\otimes 3}\right)^{S_3}$ can be using only $X^+(2,0)$. A priori, the $S_3$ orbifold appears to also require generators that are of the form $\nop{X^-(a_1,a_2)X^-(b_1,b_2)}$ however a simple calculation
$$\nop{X^-(a_1,a_2)X^-(b_1,b_2)}=\nop{X^+(a_1,b_1)X^+(a_2,b_2)}-\nop{X^+(a_1,b_2)X^+(a_2,b_1)}+\Psi$$
where $\Psi$ is a normally ordered polynomial only involving the quadratic generators $X^+(a,b)$ and $L$ and thus may be written with only $X^+(2,0)$ and $L$ finishing the proof.

\end{proof}

\section{Conformal embeddings of $\V_{c_{p,q}}$ inside $\L_{c_{r,s}}^{\otimes^n}$ }

In this part, which is somewhat disconnected from the rest of the paper, we analyze conformal embeddings 
of the form 
\begin{equation} \label{conf-mod}
\V_{c_{p,q}} \hookrightarrow \L_{c_{r,s}}^{\otimes n}.
\end{equation}
Observe that on the left hand side we have a universal VOA. If we replace  $\V_{c_{p,q}}$ with its simple quotient, things are quite simple and we have the following known result (see Wakimoto's book \cite{Wak}). 
\begin{prop}
For $n \geq 2$, the only conformal embeddings of  $\L_{c_{p,q}} \hookrightarrow \L_{c_{r,s}}^{\otimes n}$
are: $(p,q)=(2,3)$ any $n$, and $(p,q)=(3,10)$ with $n=2$.
\end{prop}
\begin{proof} Observe that $\L_{c_{2,3}}$ is the trivial $1$-dimensional vertex algebra, so the statement is clear in this case.
If there is an embedding $\L_{c_{p,q}} \hookrightarrow \L_{c_{r,s}}^{\otimes n}$, because of rationality,  
$\L_{c_{r,s}}^{\otimes n}$ decomposes as a (finite) sum of irreducible $\L_{c_{p,q}}$-modules. Recall 
the character of $\L_{c_{p,q}}$ from Section 2: 
$${\rm ch}[\L_{c_{p,q}}](\tau)=\frac{ \theta_{p,q}^{1,1} (\tau)}{\eta(\tau)}. $$
%where
%$$\theta_{p,p'}^{1,1} (\tau)=\sum_{j \in
%\mathbb{Z}}
%\left(q^{{j(jpq+q-p)}}-q^{{(jp+1)(jq+1)}}  \right). $$
Then asymptotically (as $t \to 0^+$, where $\tau=i t$):
$${\rm ch}[\L_{c_{p,q}}](\tau) \sim A_{p,q}  e^{\frac{\pi}{12 t} -\frac{6 \pi}{12 t p q}},$$
where $A_{p,q} \in \mathbb{R}$.
Now using 
$${\rm ch}[\L_{c_{r,s}}^{\otimes n}](\tau)  \sim A e^{\frac{n \pi}{12 t} -\frac{6 n \pi}{12 t p q}},$$
for some non-zero constant $A$,  and comparing the growth of characters on both sides of $\L_{c_{p,q}} \hookrightarrow \L_{c_{r,s}}^{\otimes n}$
we obtain an equality 
$$1-\frac{6}{pq}=n \left(1-\frac{6}{st}\right).$$
Easy inspection implies that the only nontrivial solution ($(p,q) \neq (2,3)$) is $(p,q)=(2,5)$ , $n=2$ 
and $(r,s)=(3,10)$.
\end{proof}

Next we consider universal embeddings   $\V_{c_{p,q}} \hookrightarrow \L_{c_{r,s}}^{\otimes n}$.
This is in fact equivalent to the following algebraic condition: 
\begin{equation} \label{dioph}
1-\frac{6(p-q)^2}{pq} = n \left(1-\frac{6(r-s)^2}{rs}\right).
\end{equation}
Next, we rewrite (\ref{dioph}) as
$$13-6\left(t+\frac1t\right)=n \left(13-6 \left(w+\frac1w \right)\right),$$
where, for convenience, we introduced  $t=\frac{p}{q}$ and $w=\frac{r}{s}$. Solving this equation in $w$ gives
$$w = \frac{ \pm \sqrt{\left(-13 n t-6 t^2+13 t-6\right)^2-144 n^2 t^2}+13 n t+6
   t^2-13 t+6}{12 n t}.$$
   As $w$ must be rational the expression under the radical must be a (rational) square. Thus, we have to find all rational points on the family of quartic curves (depending on $n$) 
   $$ u^2=\left(-13 n t-6 t^2+13 t-6\right)^2-144 n^2 t^2.$$
For fixed $n$, it is not difficult to see the above equation reduces to an elliptic curve. Since we always have a rational point (coming from $(p,q)=(r,s)=(2,3)$; $c_{p,q}=c_{r,s}=0$)  we can easily transform the quartic to its minimal cubic form.
Next we analyze $n=2$ and $n=3$ discussed in the paper.

\subsection{$n=2$}

The relevant curve is  $u^2=36 + 156 t - 335 t^2 + 156 t^3 + 36 t^4$. After a transformation it reduces to 
an elliptic curve with minimal Weierstrass' equation 
$$Y^2+Y X = X^3-1141 X+44321.$$
Using computer program {\tt Magma} we get the group of rational points ${\rm Rat}$ on this curve to be isomorphic to  $${\rm Rat} \cong \mathbb{Z}/4 \mathbb{Z} \oplus \mathbb{Z}$$ with the torsion 
part generated by $(0,6)$ (in the quartic equation) and the free part is generated by $(-6,0)$ (also in the quartic).
Using computer we can see that these solutions are quickly "huge" (large numerators and denominators). 
For instance, only solutions in the range $2 \leq p,q \leq 1000$ are (for simplicity we assume $p>q$ so we omit $(q,p)$ from the list):

$$(p,q)=(2,3), (3,10) , (4,11), (13,18), (803,852)$$
and the $(r,s)$ values are respectively
$$(r,s)=(2,3),  (2,5), ( 11,24), (9,13), (584,781)$$

\subsection{$n=3$} 

The relevant curve is 
$u^2 = 9 + 78 t - 137 t^2 + 78 t^3 + 9 t^4$ and it reduces to an elliptic curve
$$Y^2+YX = X^3-31X+3536.$$
The group of rational points ${\rm Rat}$ is again isomorphic to $\mathbb{Z}/4 \oplus \mathbb{Z}$.
The torsion subgroup is generated by an element that corresponds to  $(0,-3)$ on the quartic curve and the free 
subgroup is generated by element $(13/24, -315/64)$, also on the quartic.
Compared to $n$, the solutions are much more sparse. In the range $2 \leq p,q \leq 10,000$ with $p<q$ we have only a few pairs:
$$(p,q)=(2,3), (2,23),(13,24),(3256,4095)$$
with 
$$(r,s)=  (2,3),(9,46), (8,13), (5291,7560)$$

\begin{rem} It is interesting that for higher $n$ , the rank of the group of rational points on the elliptic curve can be bigger than $1$. 
For instance, again using {\tt Magma} it is $2$ for $n=19$ but for $n=23$ is again $1$.
 \end{rem}

We are almost done proving the following results.
\begin{prop} Among each family of rational $W$-algebras $(\L_{c_{p,q}}^{\otimes 2})^{S_2}$ and $(\L_{c_{p,q}}^{\otimes 3})^{\mathbb{Z}_3}$, there are infinitely many vertex algebras of minimal (Virasoro) central charge.
\end{prop}
\begin{proof}
We only have to argue that there are infinitely many such pairs $(p,q)$, where $p,q \geq 2$ coming from rational solutions of the elliptic curve. This follows immediately 
from the Poincare-Hurwitz theorem: on every elliptic curve with infinitely many rational points there are infinitely many rational solutions around each rational point. Since we already found several rational solutions $\frac{p}{q} >0$ for $n=2$ and $n=3$ we are done.
\end{proof}
%\section{The $c=\frac{3}{2}$ orbifold $(\L_{1/2}^{\otimes^3})^{S_3}$}

%Again we apply $\L_{1/2}=\mathcal{F}^{\ZZ_2}$, we have 
%$$\L_{1/2} \otimes \L_{1/2} \otimes \L_{1/2}=(\mathcal{F}\otimes\mathcal{F}\otimes \mathcal{F})^{\ZZ_2^3}%=((\mathcal{F}\otimes\mathcal{F}\otimes \mathcal{F})^{\ZZ_2})^{\ZZ_2^2}$$
%Then we can use the fact that 

%{\bf A: This below is not quite true (unless you work with real Lie algebras). Here we have level $2$ realization of $sl_2$ %not level one. So basically we are interested in 
%$$L_{sl_2}( 2 \Lambda_0)^{A_4 \times \mathbb{Z}_2} \cong (\L_{1/2} \otimes \L_{1/2} \otimes \L_{1/2})^{S_3},$$
%or maybe $S_4$ should be replaced with double tetrahedral group $\tilde{A}_4$. }
%$$(\mathcal{F}\otimes\mathcal{F}\otimes\mathcal{F})^{\ZZ_2}={L_{\mathfrak{so}_3}(1)} \cong V_{\sqrt{2}\ZZ}$$

%\section{The $c=-\frac{66}{5}$ orbifold $(\L_{-22/5}^{\otimes^3})^{S_3}$}

%Generators: 

%$$\omega=\omega_1+\omega_2+\omega_3$$  

%$$W=L(-6) - \frac{3800}{173} L(-2)^3 + \frac{5225}{346} L(-3)^2 + \frac{1330}{173} L(-4)L(-2)- \frac{8357}{346} \left( 
%L_1(-3)^2+L_3(-3)^2+L_3(-3)^2\right)$$

\section{Future work}

Here we announce a few results and directions for future investigation. In the sequel \cite{MPS2} we shall prove the following result:
\begin{thm} Except for finitely many central charges, the permutation orbifold 
$(\V_c^{\otimes 3})^{S_3}$ is of type $(2, 4, 6^2,8^2,9,10^2,11,12^3,14)$.
\end{thm}

Denote by $h$ and $h^\vee$ Coxeter and dual Coxeter number, respectively, of a simple Lie algebra $\bar{\frak{g}}$.
Recall \cite{Ar,Wak} that admissible levels for an affine Lie algebra $\frak{g}$ are given by
$k=-h^\vee + \frac{p}{p'}$ where $p,p' \in \mathbb{N}$, $(p,p')=1$ and 
$p \geq h^\vee$ if $(r^\vee,p')=1$ and $p \geq h$ if $(r^\vee,p')=r^\vee$. Let  
$$c(k)=rank(\bar{\frak{g}}) -\frac{12 (p' \bar{\rho} - p \bar{\rho}^\vee, p' \bar{\rho} - p \bar{\rho}^\vee)}{p p'}.$$

We also define non-degenerate admissible levels where in addition:
$p' \geq h$ if $(p',r^\vee)=1$ and $p' \geq r^\vee h^\vee$ if $(p',r^\vee)=r^\vee$. Then the main result in \cite{Ar} reads
\begin{thm}[Arakawa] If $k$ is non-degenerate admissible, then the simple affine $W$-algebra $\mathcal{W}_k(\bar{\frak{g}})$, of 
central charge $c(k)$, is rational and lisse.
\end{thm}
The family of non-degenerate admissible central charges parametrized by coprime integers $(p,p')$ are called minimal series, and the modules of the corresponding rational affine $W$-algebras are 
called minimal models.

Consider now $\bar{\frak{g}}=\frak{g}_2$, the exceptional simple Lie algebra of rank $2$. Then $r^\vee=3$, $h^\vee=4$ and $h=6$. Observe that for $p=5$ and $p'=6$ we have $c(k)=-\frac{66}{5}$ and 
$k=-4+\frac{5}{6}=-\frac{19}{6}$. However, $-\frac{19}{6}$ is not 
admissible and $\mathcal{W}_{-\frac{19}{6}}(\frak{g}_2)$ does not belong among minimal series.

%$\frak{g}_2$:  $k= -h^\vee+\frac{p}{p'}$ where $p,p' \in \mathbb{N}$, $(p,p')=1$ and 
%$p \geq h^\vee$ if $(r^\vee,q)=1$ and $p \geq h$ is $(r^\vee,q)=r^\vee$, where 

%$p' \geq h$ and $p \geq h^\vee$ (here $h^\vee=4$ and $h=6$).
%Then the central charge of $W^k(\frak{g}_2,f)$ is:
%$$c(k)=2 -\frac{12 (p' \bar{\rho} - p \bar{\rho}^\vee, p' \bar{\rho} - p \bar{\rho}^\vee)}{p p'}.$$
%For $k=-4+\frac56=-\frac{19}{6}$, 

In the sequel \cite{MPS2}, we shall prove an isomorphism 
$$(\L_{-22/5}^{\otimes^3})^{S_3} \cong \mathcal{W}_{-\frac{19}{6}}(\frak{g}_2,f_{princ}),$$
using explicit generators of $\mathcal{W}^k(\frak{g}_2,f_{princ})$. This in particular, proves
rationality and $C_2$-cofiniteness (after \cite{CarMiy,Miy}) of an affine $W$-algebra outside the usual rational series.
Moreover, this gives an elegant expression for the character of the affine $W$-algebra (see Section 2).

We also record a related conjecture which is currently beyond our reach.
\begin{conjecture} We have an isomorphism of vertex algebras
$$(\L_{-22/5}^{\otimes^4})^{S_4} \cong \mathcal{W}_{-\frac{49}{6}}(\frak{f}_4,f_{princ}).$$
\end{conjecture}
This conjecture if motivated by some computational evidence that $(\L_{-22/5}^{\otimes^4})^{S_4}$ is of type 
 $(2,6,8,12)$, which is the type of the right-hand side.
For the exceptional Lie algebra $\frak{f}_4$, we have $h=12$ and $h^\vee=9$.
Again, using the above formula $c(k)=-\frac{88}{5}$ for $(p,p')=(5,6)$ and $k=-\frac{49}{6}$, which does not belong to admissible series. 

Observe that our Theorem \ref{rank2simple}, part (a) and (e), gives an isomorphism 
\begin{equation}
( \L_{-22/5}^{\otimes^4})^{D_4} \cong (\L_{-44/5}^{\otimes^2})^{\mathbb{Z}_2},
\end{equation}
so the left-hand side is of type $(2,4,6,8)$.

In another recent development, Li, one of the authors, and Wauchope \cite{LMW} proved the following result for the Heisenberg-Virasoro $\mathcal{HV}_{c}$, $N=1$ and $N=2$ superconformal vertex algebras $\V^{N1}_c$, $\V^{N2}_c$.
\begin{thm} We have
\begin{itemize}  
\item[(1)] 
The orbifold ($\mathcal{HV}_{c}^{\otimes 2})^{S_2}$ is 
is of type $(1, 2^2,3,4^3,5)$.
\item[(2)] Generically, the $S_2$-orbifold of $(\V^{N1}_c)^{\otimes 2}$
is of type $(\frac{3}{2},2,\frac{7}{2},4^2,\frac{9}{2},\frac{11}{2},6^2,\frac{13}{2})$.
\item[(3)] Generically, the $S_2$-orbifold of $(\V^{N2}_c)^{\otimes 2}$ is of type 
$(1,{\frac{3}{2}}^2,2^2,{\frac{5}{2}}^2,3^2,{\frac{7}{2}}^4,4^6,{\frac{9}{2}}^4,5^2)$.
\end{itemize}
\end{thm}
Jointly with Li, we are planning to study $\mathbb{Z}_3$ and $S_3$-orbifolds of these vertex algebras.

\vspace{.3in}

\vspace{.2in}

\end{document}